\documentclass[11pt]{amsart}
\usepackage{amsfonts, amssymb, amsmath, mathrsfs,mathtools}
\usepackage{stmaryrd}
\usepackage[all]{xy}
\usepackage{tikz-cd}
\usepackage[pdftex, colorlinks=false, pdfborder={0 0 0}]{hyperref}
\hypersetup{
	pdftitle={Crystals and the parabolicity conjecture},
	pdfauthor={Marco D'Addezio},
	pdfsubject={Course notes on crystals, F-isocrystals, and parabolicity},
	pdfkeywords={crystals, p-divisible groups, F-isocrystals, slope filtration, monodromy, Hecke orbits}
}
\usepackage{url}
\newcommand{\spa}{\medskip}

\tikzset{
	labl/.style={anchor=south, rotate=90, inner sep=.5mm}
}

\usepackage[utf8]{inputenc}

\usepackage{bbm}

\usepackage{relsize}
\usepackage[bbgreekl]{mathbbol}
\DeclareSymbolFontAlphabet{\mathbb}{AMSb}
\DeclareSymbolFontAlphabet{\mathbbl}{bbold}

\newcommand{\F}{\mathbb{F}}

\newcommand{\Z}{\mathbb{Z}}

\newcommand{\spec}{\mathrm{Spec}}

\newcommand{\Foi}{\Fisoc^\mathrm{\dag}}

\newcommand{\iso}{\xrightarrow{\sim}}

\newcommand{\Zl}{\mathbb{Z}_{\ell}}
\newcommand{\Ql}{\mathbb{Q}_{\ell}}

\newcommand{\Qp}{\mathbb{Q}_p}
\newcommand{\Qlbar}{\overline{\mathbb{Q}}_{\ell}}

\newcommand{\kbar}{{\bar{k}}}

\newcommand{\Zp}{{\mathbb{Z}_p}}

\newcommand{\mr}{\mathrm}

\newcommand{\Fp}{\F_p}

\newcommand{\calA}{\mathcal{A}}

\newcommand{\calE}{\mathcal{E}}

\newcommand{\calF}{\mathcal{F}}

\newcommand{\calH}{\mathcal{H}}

\newcommand{\calO}{\mathcal{O}}

\newcommand{\calL}{\mathcal{L}}

\newcommand{\calM}{\mathcal{M}}

\newcommand{\calN}{\mathcal{N}}

\newcommand{\calX}{\mathcal{X}}

\newcommand{\ovfp}{\overline{\mathbb{F}}_p}

\newcommand{\Spec}{\mathrm{Spec}}

\newcommand{\Gal}{\mathrm{Gal}}
\newcommand{\Hom}{\mathrm{Hom}}

\newcommand{\Ker}{\mathrm{Ker}}

\newcommand{\GL}{\mathrm{GL}}

\newcommand{\Fisoc}{\mathrm{F\textrm{-}Isoc}}

\newcommand{\Spf}{{\mathrm{Spf}}}

\newcommand{\stack}[1]{\cite[\href{https://stacks.math.columbia.edu/tag/#1}{Tag #1}]{Stacks}}

\usepackage{epigraph}

\usepackage{etoolbox}

\makeatletter
\patchcmd{\epigraph}{\@epitext{#1}}{\itshape\@epitext{#1}}{}{}
\makeatother

\begin{document}
	\bibliographystyle{alpha}
	\newtheorem{theo}{Theorem}[subsection]
	\newtheorem{hope}[theo]{Hope}
	\newtheorem{prop}[theo]{Proposition}
	\newtheorem{lemm}[theo]{Lemma}
	\newtheorem{coro}[theo]{Corollary}
	\newtheorem{claim}[theo]{Claim}
	\newtheorem{fact}[theo]{Fact}
	\newtheorem*{conj*}{Conjecture}
	\newtheorem{conj}[theo]{Conjecture}
	
	\theoremstyle{definition}
	\newtheorem{defi}[theo]{Definition}
	\newtheorem{ques}[theo]{Question}
	\newtheorem{prob}[theo]{Problem}

	\newtheorem{nota}[theo]{Notation}
	\newtheorem*{theo*}{Theorem}
	\newtheorem*{ques*}{Question}
	\newtheorem{property}[theo]{Property}
	\newtheorem{hypo}[theo]{Hypothesis}
	
	\newtheorem{answer}[theo]{Answer}
	\newtheorem{rema}[theo]{Remark}
	\newtheorem{exam}[theo]{Example}
	\newtheorem{warning}[theo]{Warning}
	\newtheorem{notation}[theo]{Notation}
	\newtheorem{cons}[theo]{Construction}
	
	\numberwithin{equation}{section}
	
	\title{Crystals and the parabolicity conjecture}
		
	\author{Marco D'Addezio}
	\address{IRMA, Université de Strasbourg, 67000 Strasbourg, France}
	\email{daddezio@unistra.fr}
	\date{October 11, 2026}
	\thanks{Course notes for the workshop ``A Sprint on Local Systems'', Wuppertal 2026, website: \url{https://sites.google.com/view/workshop-ls2026}}
	\begin{abstract}
		These notes introduce crystals and crystalline Dieudonné theory through the example of abelian varieties. After discussing $p$-divisible groups, $F$-isocrystals, slope filtrations, and overconvergence, we explain the parabolicity theorem for $F$-isocrystals and some of its applications to monodromy and Hecke orbits. Two appendices formulate research projects on isocrystals over formal schemes and on monodromy beyond the constant-slope setting. The exposition is intended for readers familiar with basic algebraic geometry and the étale and fppf topologies; prior knowledge of Witt vectors is helpful but not required.
	\end{abstract}
	\maketitle
	
	\tableofcontents

	\section*{Introduction}

	These notes grew out of lectures at the workshop ``A Sprint on Local Systems'', Wuppertal 2026. Their aim is to give a concrete route from the de Rham cohomology of elliptic curves to crystalline Dieudonné theory and, ultimately, to the parabolicity theorem for $F$-isocrystals. The emphasis is on examples and geometric intuition; several proofs are presented only as sketches or exercises, with references to the literature for complete arguments.

	\spa

	Lecture~1 introduces crystals and their description by flat connections. Lecture~2 discusses $p$-divisible groups, $F$-isocrystals, slope filtrations, overconvergence, and monodromy. Lecture~3 states the parabolicity theorem, sketches its proof through $\dagger$-hulls, and describes applications to local monodromy and Hecke orbits. The appendices collect two research projects.

	\section{Crystals}
	\subsection{Elliptic curves and Dieudonné modules}

Let $k$ be a field with algebraic closure $\kbar$ and let $E$ be an elliptic curve over $k$. The set of $\kbar$-points of $E$, denoted by $E(\kbar)$, is endowed with a commutative group structure. For a prime number $\ell$ different from the characteristic of $k$ and for every $n\geq 0$ there exists a non-canonical isomorphism  $$E(\kbar)[\ell^n]\simeq\Z/\ell^{n}\oplus \Z/\ell^{n}.$$ This agrees with the situation over the complex numbers, where elliptic curves can be written as quotients $\mathbb{C}/\Lambda$, where $\Lambda$ is a lattice of $\mathbb{C}$. In particular, note that $E(\kbar)[\ell^n]$, as an abstract group, does not depend on the choice of $E$ and it has the same structure when $\ell$ varies. The projective limit induced by the multiplication-by-$\ell$ maps $$T_\ell(E)\coloneqq \varprojlim_n E(\kbar)[\ell^n]$$ is the \textit{$\ell$-adic Tate module} of $E$. It is non-canonically isomorphic to $\Zl\oplus \Zl$ and it behaves like a homology group of the curve. It is in fact the dual of the first $\ell$-adic cohomology group of $E_\kbar$. The absolute Galois group $\Gal(\kbar/k)$ acts on $T_\ell(E)$, giving rise to an $\ell$-adic representation $$\rho_{E,\ell}\colon \Gal(\kbar/k)\to \GL(T_\ell(E))\simeq \GL_2(\Zl).$$ These representations are fundamental objects in number theory.

\spa

A different story happens for $E(\kbar)[p^n]$ when $k$ is of characteristic $p$. In this case, the group is a free $\Z/p^n$-module of rank $1$ or $0$, depending on $E$. This rank, which is independent of $n$, is the \textit{$p$-rank} of $E$. The set of all elliptic curves over $k$ then splits into two classes. The ones of $p$-rank $1$ are called \textit{ordinary}, while the ones of $p$-rank $0$ are called \textit{supersingular}.\footnote{This term is not related to the notion of singularity of an algebraic variety; elliptic curves are always smooth by definition.} 

\spa

To better understand this phenomenon, it is very convenient to look at the entire group schemes $$E[p^n]\coloneqq \mathrm{Ker}(E\xrightarrow{\cdot p^n} E).$$
These are finite group schemes over $k$ of order $p^{2n}$, matching the $\ell$-primary setting. This order is the same for ordinary and supersingular elliptic curves. These group schemes, unlike the $\ell$-primary situation, are not reduced. One of the aims of this course is to convince you that this is not just an annoying pathology, but it is rather an extra feature that gives more information about the elliptic curve than the $\ell$-primary torsion. In particular, we will see that these group schemes are the right objects to look at to study the deformations of $E$.

\spa
Let us dive more into the structure of $E[p^n]$. The multiplication-by-$p$ map on $E$ is faithfully flat and it decomposes as $p=VF$, where $F\colon E\to E^{(p)}$ is the {relative Frobenius} and $V\colon E^{(p)}\to E$ is the Verschiebung. We have an exact sequence 
 $$0\to E^{(p)}[V]\to E^{(p)}[p]\xrightarrow{V} E[F]\to 0.$$
The Frobenius is purely inseparable, which implies that $E[F^n]$ is a connected group scheme of order $p^n$ (with no non-trivial $\kbar$-points). It follows that the $p$-rank of $E$ is the same as the rank of $E^{(p)}[V](\kbar)$ over $\Z/p$. In turn, we deduce that $E$ is supersingular if and only if $V$ is purely inseparable. Since $V$ is of degree $p$, it is enough to check whether it vanishes at the tangent space of the identity. In turn, since $E$ is parallelisable, we have just proved that $E$ is supersingular if and only if  $$V^*\colon H^0(E,\Omega^1_{E/k})\to H^0(E,\Omega^1_{E/k})^{(p)}$$ vanishes. It is then natural to look at the behaviour of $F$ and $V$ on the entire {de Rham cohomology} of $E$.


\subsubsection{De Rham cohomology}

We first recall the definition of algebraic de Rham cohomology.

\begin{defi}
Let $X\to S$ be a smooth morphism of schemes. The \textit{algebraic de Rham cohomology groups} of $X$ over $S$ are defined as the hypercohomology groups $$H^i_{\mathrm{dR}}(X/S)\coloneqq \mathbb{H}^i(X,\Omega^\bullet_{X/S})$$ where $\Omega^\bullet_{X/S}$ is the algebraic de Rham complex of $X$ over $S$.
\end{defi} 

\begin{rema}If $S$ is the spectrum of a field and $X$ is smooth and proper over $k$, these groups are finite dimensional $k$-vector spaces. In addition, if $k=\mathbb{C}$, Grothendieck proved that there is a canonical isomorphism $$H^i_{\mathrm{dR}}(X/\mathbb{C})= H^i_{\mathrm{sing}}(X(\mathbb{C}),\mathbb{C}).$$ In this case, the algebraic de Rham cohomology groups can be seen as algebraic incarnations of the singular cohomology groups.
\end{rema}

Coming back to our setting, as it happens over the complex numbers, the de Rham cohomology group $H^1_{\mathrm{dR}}(E/k)$ is a $2$-dimensional $k$-vector space. In addition, $H^1_{\mathrm{dR}}(E/k)$ admits the \textit{Hodge filtration}\footnote{Recall that in positive characteristic the Hodge-to-de Rham spectral sequence does not degenerate at the first page. For example, there exist surfaces such that the $E_1$-differential $d\colon H^0(X,\Omega^1_{X/k})\to H^0(X,\Omega^2_{X/k})$ does not vanish \cite{Mum61}. The spectral sequence does degenerate at the first page, for example, when $X$ is an abelian variety, a K3 surface, or a complete intersection \cite{Wed08}.} $$\mathrm{Fil}^1_\mathrm{Hdg}=H^0(E,\Omega^1_{E/k})\subseteq H^1_{\mathrm{dR}}(E/k)$$ with quotient $$H^1_{\mathrm{dR}}(E/k)/\mathrm{Fil}^1_\mathrm{Hdg}= H^1(E,\calO_{E}).$$ Both $H^0(E,\Omega^1_{E/k})$ and $H^1(E,\calO_{E})$ are $1$-dimensional $k$-vector spaces. The cohomology group $H^1_{\mathrm{dR}}(E/k)$ admits also the \textit{conjugate filtration,} which thanks to the \textit{Cartier isomorphism} $$C^{-1}\colon \Omega^i_{E^{(p)}/k}\xrightarrow{\sim} \calH^i(F_*\Omega^\bullet_{E/k})$$ can be written as $$\mathrm{Fil}^\mathrm{conj}_1=H^1(E,\calO_{E})^{(p)}\subseteq  H^1_{\mathrm{dR}}(E/k).$$ In this case, taking the quotient yields $$H^1_{\mathrm{dR}}(E/k)/\mathrm{Fil}^\mathrm{conj}_1=H^0(E,\Omega^1_{E/k})^{(p)}.$$
The Frobenius and Verschiebung induce maps $$F^*\colon H^1_{\mathrm{dR}}(E/k)^{(p)}\to H^1_{\mathrm{dR}}(E/k)\quad \textrm{and}\quad V^*\colon H^1_{\mathrm{dR}}(E/k)\to H^1_{\mathrm{dR}}(E/k)^{(p)}.$$ These maps respect the filtrations and are dual to each other via Poincaré duality. More precisely, pullback by the relative Frobenius induces the following exact sequence of complexes on $E^{(p)}$:
\[
0\longrightarrow \Omega^1_{E^{(p)}/k}[-1]
\longrightarrow \Omega^\bullet_{E^{(p)}/k}
\xrightarrow{F^*}\calH^0(F_*\Omega^\bullet_{E/k})[0]
\longrightarrow 0.
\]
Taking hypercohomology and using the Cartier isomorphism, we deduce that the Hodge and conjugate filtrations are related by the isomorphism
\begin{equation}\label{eq:FZip}
 H^1_{\mathrm{dR}}(E/k)^{(p)}/{(\mathrm{Fil}^1_\mathrm{Hdg})^{(p)}}\xrightarrow{\sim} \mathrm{Fil}^\mathrm{conj}_1\subseteq H^1_{\mathrm{dR}}(E/k).
\end{equation}

 We have the following more complete characterisation of the property of being supersingular.
\begin{prop}\label{Prop:Supersingular}
	Let $E$ be an elliptic curve over a field $k$ of characteristic $p$. The following properties are equivalent.
	\begin{enumerate}
		\item $E$ is supersingular.
	\spa
		\item The Verschiebung $V\colon E^{(p)}\to E$ is purely inseparable.
	\spa
		\item $V^*\colon  H^0(E,\Omega^1_{E/k})\to H^0(E,\Omega^1_{E/k})^{(p)}$ is the zero morphism.
	\spa
		\item $F^*\colon  H^1(E,\calO_{E})^{(p)}\to H^1(E,\mathcal{O}_{E})$ is the zero morphism.
	\spa
		\item The Frobenius on $H^1_{\mathrm{dR}}(E/k)$ is nilpotent of exact index $2$: explicitly, $F^*\circ(F^*)^{(p)}=0$ and $F^*\neq 0$.
	\spa
		\item The Hodge and conjugate filtrations on $H^1_{\mathrm{dR}}(E/k)$ coincide.
	\end{enumerate}
\end{prop}
\begin{proof}
We have already discussed the equivalence of the first three properties. Thanks to Poincaré duality (3) and (4) are equivalent. Property (4) implies that $F^*\circ(F^*)^{(p)}=0$. To exclude index $1$ we note that $F^*$ cannot vanish thanks to \eqref{eq:FZip}. Thus $(4)\Rightarrow (5)$. Property (5) implies $\mathrm{Im}((F^*)^{(p)})=\mathrm{Ker}(F^*)$, which in turn implies $(6)$ thanks to \eqref{eq:FZip}. Finally, {if (6) holds, then $\mathrm{Im}(F^*)=\mathrm{Fil}^{\mathrm{conj}}_1=\mathrm{Fil}^1_\mathrm{Hdg}$; hence the induced map to $H^1_{\mathrm{dR}}(E/k)/\mathrm{Fil}^1_\mathrm{Hdg}=H^1(E,\calO_E)$ vanishes, which is (4).}\end{proof}
\subsubsection{Torsion Dieudonné modules}

 Proposition \ref{Prop:Supersingular} shows that the $k$-vector space $H^1_{\mathrm{dR}}(E/k)$ together with $F^*$ and $V^*$ encodes important information about the $p$-torsion of $E$. This is the first example of a \textit{torsion Dieudonné module}. Let us present torsion Dieudonné modules when $k$ is perfect.  
 
 \begin{defi}We define $$\mathcal{D}_k\coloneqq W\{F,V\}/(FV-p,VF-p),$$ where $W$ is the ring of Witt vectors of $k$ with Frobenius lift $\sigma$ and $F,V$ are two non-commuting variables satisfying $F\cdot a=\sigma(a)\cdot F$ and $V\cdot \sigma(a)=a\cdot V$ for every $a\in W$. The ring $\mathcal{D}_k$ is called the \textit{Dieudonné ring} of $k$. A finite $W$-length left $\mathcal{D}_k$-module is a \textit{torsion Dieudonné module over $k$}. 
\end{defi}
\begin{exam}The groups $H^1_\mathrm{dR}(E/k),$ $H^0(E,\Omega^1_{E/k})$, and $H^1(E,\calO_E)$ are torsion Dieudonné modules over $k$ with the action of $F$ and $V$ given by the maps $F^*$ and $V^*$ defined above. Moreover, if $G$ is a commutative finite group scheme over $k$ and $\omega_{G/k}$ is pullback of the sheaf of differentials $\Omega^1_{G/k}$ to the identity section of $G$, the Frobenius and Verschiebung of $G$ induce maps $$F^*\colon \omega_{G/k}^{(p)}\to \omega_{G/k}\quad \textrm{and}\quad V^*\colon \omega_{G/k}\to \omega_{G/k}^{(p)}$$ which make $\omega_{G/k}$ into a torsion Dieudonné module over $k$. If $M$ is a torsion Dieudonné module over $k$, the dual $M^\vee\coloneqq \mathrm{Hom}_W(M,K/W)$, where $K$ is the fraction field of $W$, is also a torsion Dieudonné module over $k$ with the action of $F$ and $V$ given by $(F\cdot f)(x)\coloneqq \sigma(f(V\cdot x))$ and $(V\cdot f)(x)\coloneqq \sigma^{-1}(f(F\cdot x))$.
\end{exam}
We are now ready to state the main result of this lecture.
\begin{theo}[Dieudonné, Cartier, Oda]\label{Thm:DieudonneOda}
	Let $k$ be a perfect field of characteristic $p$. Denote by $\mathsf{FGS}_{p/k}$ the category of commutative finite group schemes over $k$ of $p$-power order and by $\mathsf{DMod}^{\mathrm{tor}}_k$ the category of torsion Dieudonné modules over $k$. There is an exact equivalence of abelian categories
	\[
	\mathbb{D}\colon \mathsf{FGS}_{p/k}^{\mathrm{op}}\xrightarrow{\sim}\mathsf{DMod}^{\mathrm{tor}}_k
	\]
	satisfying the following properties.
	\begin{enumerate}
		\item $G$ has order $p^n$ if and only if $\mathrm{length}_{W}(\mathbb{D}(G))=n$.
		
		\spa

		\item $\mathbb{D}(G)/F(\mathbb{D}(G))={\omega_{G/k}}$.
		
		\spa
		
		\item $\mathbb{D}(G^\vee)=\mathbb{D}(G)^\vee,$ where $G^\vee$ is the Cartier dual of $G$.
		
		\spa
		\item For an abelian variety over $k$, there is a canonical isomorphism $$\mathbb{D}(A[p])= H^1_{\mathrm{dR}}(A/k)$$ as $\mathcal{D}_k$-modules.
	\end{enumerate}

\end{theo}

\begin{exam}\label{exam:Dieudonne}Here are some examples of Dieudonné modules.

	\begin{itemize}
		\item[--] $\mathbb{D}(\mu_p)=\omega_{\mu_p/k}$ is isomorphic to $k$ with $F=0$ and $V=\sigma^{-1}$.
		\spa
	\item[--] $\mathbb{D}(\Z/p)=\mathbb{D}(\mu_p)^\vee$ is isomorphic to $k$ with $F=\sigma$ and $V=0$.
	\spa
	\item[--] $\mathbb{D}(\alpha_p)=\omega_{\alpha_p/k}$ is self-dual, thus it is isomorphic to $k$ with $F=0$ and $V=0$.
	\spa
	\item[--] If $E$ is ordinary and $k$ is algebraically closed, then  $\mathbb{D}(E[p])=k\oplus k$ with $$F=\begin{pmatrix} \sigma & 0 \\ 0 & 0 \end{pmatrix}\quad \textrm{and} \quad V=\begin{pmatrix} 0 & 0 \\ 0 & \sigma^{-1} \end{pmatrix}.$$ The two summands correspond to $$\mathrm{Fil}_1^\mathrm{conj}\subseteq H^1_{\mathrm{dR}}(E/k)\quad \textrm{and}\quad \mathrm{Fil}^1_\mathrm{Hdg}\subseteq H^1_{\mathrm{dR}}(E/k).$$ In this case, $E[p]=\mu_p\times \Z/p$.
	\spa
	\item[--] If $E$ is a supersingular elliptic curve over $k$, putting together Proposition \ref{Prop:Supersingular} and Theorem \ref{Thm:DieudonneOda}, we get that $\mathbb{D}(E[p])=k\oplus k$ with $F$ and $V$ nilpotent of index $2$. We deduce the exact sequence $$0\to \mathbb{D}(\alpha_p) \to \mathbb{D}(E[p])\to {\mathbb{D}}(\alpha_p)\to 0,$$ which shows that $E[p]$ is a non-trivial extension of $\alpha_p$ by $\alpha_p$.
	\end{itemize} 
\end{exam}

When $k$ is not perfect, for example when $k=\Fp(t)$, a variant of Theorem \ref{Thm:DieudonneOda} is still valid. Nonetheless, in this situation, the action of $F$ and $V$ on $H^1_\mathrm{dR}(E/k)$ is not enough to recover $E[p]$. One needs to further endow the modules with a \textit{connection}. Let us explain where this connection comes from.

\subsubsection{Gauss--Manin connection for elliptic curves.} Relative de Rham cohomology of a smooth and proper morphism is endowed with the \textit{Gauss--Manin connection}. For elliptic curves, this connection is quite explicit. If we write $f\colon E\to \Spec(k)$ for the structure morphism, there is a natural exact sequence of complexes (exercise) \begin{equation}\label{eq:GaussManin}0\to f^*\Omega^1_{k/\Fp}\otimes \Omega^\bullet_{E/k}[-1] \to K^\bullet \to \Omega^\bullet_{E/k} \to 0,\end{equation}
 where  $$K^\bullet\coloneqq  \calO_E\to \Omega^1_{E/\Fp}\to \Omega^2_{E/\Fp}/f^*\Omega^2_{k/\Fp}\to 0\to \cdots.$$

After applying $Rf_*$ to \eqref{eq:GaussManin},  the connecting morphism $$\nabla_{\mathrm{GM}}\colon R^1f_*\Omega^\bullet_{E/k}\to R^1f_*(f^*\Omega^1_{k/\Fp}\otimes \Omega^\bullet_{E/k})=\Omega^1_{k/\Fp}\otimes R^1f_*\Omega^\bullet_{E/k}$$ is a flat connection on the vector bundle $R^1f_*\Omega^\bullet_{E/k}=H^1_{\mathrm{dR}}(E/k)$. This is called the \textit{Gauss--Manin connection} on $H^1_{\mathrm{dR}}(E/k)$. {If $k$ has a finite $p$-basis,} the datum of $(H^1_{\mathrm{dR}}(E/k),F,V,\nabla_\mathrm{GM})$ is enough to recover $E[p]$ as a group scheme over $k$. From this description, it is not hard to see that the Gauss--Manin connection preserves the conjugate filtration.

\begin{lemm}
The conjugate filtration on $H^1_{\mathrm{dR}}(E/k)$ is preserved by the Gauss--Manin connection.
\end{lemm}
\begin{proof}
	Exercise. \textit{Hint: Look at $F^*\colon (K^\bullet)^{(p)}\to F_*K^\bullet$.}
\end{proof}

In order to study more in general the $p^n$-torsion of abelian schemes over general bases (and other $p$-torsion and $p$-adic phenomena in algebraic geometry) Grothendieck introduced the theory of \textit{crystalline cohomology} and \textit{crystals}, which provides a $p$-adic lift of de Rham cohomology and Gauss--Manin connections to $\Zp$ and allows extensions of Dieudonné theory to more general bases. We will give a quick introduction to this theory in the next sections.

\subsection{The crystalline site} Grothendieck's point of view on de Rham cohomology is that flat connections are vector bundles endowed with some infinitesimal descent datum. With this viewpoint, Taylor series become a way to describe the relation between two infinitely close points. Crystals are a similar notion, where the infinitesimal neighborhoods are replaced by \textit{PD-thickenings}, which are better suited to deal with Taylor series. Indeed, they solve the lack of operations like $\tfrac{x^n}{n!}$ over general rings. 
	
	\begin{defi}
		A \textit{PD-structure} or \textit{divided power structure} on an ideal $I$ of a ring $A$ is a collection of maps $\gamma_n: I \to I$ for $n \geq 1$  satisfying the following axioms for all $a\in A$, $x,y\in I,$  and $m,n\geq 1$, with the convention that $\gamma_0(x)\coloneqq 1$.
		
		\begin{enumerate}
			\item $\gamma_1(x)=x$.  
			\item $\gamma_n(x+y)=\sum_{i=0}^n \gamma_i(x)\gamma_{n-i}(y)$.
			\item $\gamma_n(ax)=a^n\gamma_n(x)$.
			\item $\gamma_m(x)\gamma_n(x)=\binom{m+n}{m}\gamma_{m+n}(x)$.
			\item $\gamma_m(\gamma_n(x))=\frac{(mn)!}{m!(n!)^m}\gamma_{mn}(x)$. 
		\end{enumerate}

		A \textit{crystalline extension\footnote{In the literature this corresponds to the definition of an \textit{affine PD-thickening} over $(\Z_{(p)},(p))$. We prefer this more compact notation.}} is the datum of an $\Fp$-algebra $A$, a surjective morphism of $\Z_{(p)}$-algebras $B\twoheadrightarrow A$ with $p\in \mathrm{Nil}(B)$,
and a PD-structure $\gamma$ on $\mr{Ker}(B\twoheadrightarrow A)$ such that $\gamma_n(p)=p^n/n!$ for every $n\geq 0$.

\spa

If $X$ is a scheme over $\Fp$, we denote by $X_\mr{cris}$ the opposite of the category of \textit{crystalline extensions} endowed with an open immersion $\Spec(A)\hookrightarrow X.$ We endow it with the Zariski topology. The site $X_\mr{cris}$ is the \textit{(small affine) crystalline site}\footnote{In what follows, if $X=\Spec(C)$ is affine, one can replace $X_\mr{cris}$ with the smaller site of crystalline extensions  $(B\twoheadrightarrow A,\gamma)$ such that $A=C$.} of $X$. The structure sheaf $\calO_\mathrm{cris}$ of $X_\mr{cris}$ is defined by $\calO_\mathrm{cris}(B\twoheadrightarrow A,\gamma)\coloneqq B$. The cohomology groups $H^i(X_\mr{cris},\calO_\mathrm{cris})$ are called the \textit{crystalline cohomology groups} of $X$. The absolute Frobenius morphism $F\colon X\to X$ induces a morphism of sites $F\colon X_\mr{cris}\to X_\mr{cris}$ which endows crystalline cohomology groups with a fundamental Frobenius action. See \cite{Ber74} and \stack{07GI} for more details.
	\end{defi}

	\begin{rema}
		If $(B\twoheadrightarrow A,\gamma)$ is a crystalline extension, then $\mathrm{Ker}(B\twoheadrightarrow A)\subseteq \mathrm{Nil}(B)$. Note also that for every $\Fp$-algebra $A$ and a surjective morphism $B\twoheadrightarrow A$ such that $\mathrm{Ker}(B\twoheadrightarrow A)^2=0$, there exists a trivial PD-structure such that $\gamma_n^\mathrm{triv}(x)=0$ for every $x\in \mathrm{Ker}(B\twoheadrightarrow A)$ and $n\geq 2$. If $p$ is odd, $\gamma^\mathrm{triv}$ satisfies in addition the condition $\gamma_n^\mathrm{triv}(p)=p^n/n!$ for every $n\geq 0$, giving rise to crystalline extensions. For higher order thickenings, the existence of a PD-structure is not guaranteed.
	\end{rema}
	
	\begin{defi} If $\calF$ is a sheaf of $\calO_\mathrm{cris}$-modules on $X_\mr{cris}$, we denote by $\calF_B$ the $B$-module $\calF(B\twoheadrightarrow A,\gamma)$.  A finitely presented quasi-coherent $\calO_\mathrm{cris}$-module $\calF$ over $X_\mr{cris}$ is a \textit{crystal} if it satisfies the following rigidity condition. For any morphism $(B'\twoheadrightarrow A',\gamma') \to (B\twoheadrightarrow A,\gamma)$ in $X_\mr{cris}$, the comparison map $$\calF_B\otimes_B B'\to \calF_{B'}$$ is an isomorphism. We denote by $\mathrm{Crys}(X)$ the category of crystals over $X$. This is a $\Zp$-linear abelian category. The category $\mathrm{Isoc}(X)$ of \textit{isocrystals} over $X$ is the category $\mathrm{Crys}(X)\otimes_\Zp \Qp$ obtained by keeping the objects unchanged and setting $$\Hom_{\mathrm{Isoc}(X)}(\calM,\calN) \coloneqq \Hom_{\mathrm{Crys}(X)}(\calM,\calN)\otimes_\Zp \Qp$$ for every $\calM,\calN\in \mathrm{Crys}(X)$.
	\end{defi}

	\begin{exam}
		If $X = \Spec(R)$ for a perfect ring $R$, then (finitely presented) crystals over $X$ are equivalent to finitely presented modules over the ring of Witt vectors $W(R)$. Another situation where crystals can be easily described is when $R$ is semi-perfect, i.e. the absolute Frobenius morphism $R\to R$ is surjective. In this case, crystals over $\Spec(R)$ are equivalent to finitely presented modules over $\mathbb{A}_\mathrm{cris}(R)$.
	\end{exam}

	\subsection{Crystals as flat connections}
	On a smooth scheme over a perfect field, crystals can be locally described in a more concrete way using modules with integrable connections. 

	\spa
	
	Let $A$ be a smooth $k$-algebra such that $\Omega^1_{A/k}$ is a free $A$-module with basis $dx_1,\dots,dx_n$. Let $\tilde{A}$ be a $p$-complete flat lift of $A$ to $W$. Choose lifts $\tilde{x}_i\in \tilde{A}$ of $x_i\in A$ and let $$\widehat{\Omega}^1_{\tilde{A}/W}=\bigoplus_{i=1}^n \tilde{A} d\tilde{x}_i$$ be the $p$-adically completed sheaf of differentials. For a finitely generated module $M$ with flat connection
		$$
		\nabla: M \to M \otimes_{\tilde{A}}\widehat{\Omega}^1_{\tilde{A}/W},
		$$ there exist unique derivations $\nabla_{\partial_{x_i}}: M \to M$ such that $\nabla(m) = \sum_i \nabla_{\partial_{x_i}}(m) \otimes d\tilde{x}_i$. These are the \textit{covariant derivatives} of $\nabla$ with respect to $\partial_{x_1},\dots,\partial_{x_n}$.
	
	\begin{defi}
		A (finitely generated) \textit{topologically quasi-nilpotent flat connection} over $\tilde{A}$ is a finitely generated $\tilde{A}$-module $M$ equipped with a flat connection
		$$
		\nabla: M \to M \otimes_{\tilde{A}}\widehat{\Omega}^1_{\tilde{A}/W}
		$$
		 such that for every $m\in M$, the reduction modulo $p$ of $\nabla_{\partial_{x_i}}$ is a nilpotent $A^p$-linear endomorphism of $M/pM$ for every $i$.
	\end{defi}

	\begin{theo}\label{Thm:CrystalsAsConnections}
		The category of (finitely generated) crystals over $\Spec(A)$ is equivalent to the category of (finitely generated) \textit{topologically quasi-nilpotent flat connections} over $\tilde{A}$.
	\end{theo}
\subsubsection*{Idea of the proof}  Let $B\coloneqq A[\epsilon]$. A derivation $\partial\colon A\to A$ corresponds to a morphism of rings $A\to B$ defined by $a\mapsto a+\partial(a)\epsilon$. We write $$f_\partial \colon (B\twoheadrightarrow A,\gamma^\mathrm{triv})\to (A\twoheadrightarrow A,\gamma^\mathrm{triv})$$ for the associated morphism in the crystalline site of $\Spec(A)$. For a crystal $\calM$, write $M\coloneqq\calM(A\twoheadrightarrow A,\gamma^\mathrm{triv})$ and $N\coloneqq\calM(B\twoheadrightarrow A,\gamma^\mathrm{triv})$. By applying the crystal property for $f_0$ and $f_\partial$ (where $_0$ stands for the $0$-derivation) we get a $B$-linear isomorphism $$t_\partial\colon f^*_\partial M\iso N \iso f^*_0 M=M[\epsilon].$$ The connection attached to $\calM$, when reduced modulo $p$, is of the form $(M,\nabla)$. For $m\in M$ and $\partial$ a derivation, we have the fundamental relation $$t_\partial(m\otimes 1)=m+\nabla_\partial(m)\epsilon.$$ The identity $am\otimes 1=m\otimes (a+\partial(a)\epsilon)$ in $f^*_\partial M$ gives the Leibniz rule $$\nabla_\partial(am)=a\nabla_\partial(m)+\partial(a)m.$$  This makes explicit the idea that a connection is a rule that tells you how to deform elements of a module when moving between infinitely close points. The PD-structures are then needed to pass from first order thickenings to higher order thickenings via the theory of Taylor expansions.\qed
	
\begin{exam}[Kummer connection]
		Let $X = \Spec(\F_p[x,x^{-1}])$.  A $p$-adic lift is $\calX = \Spf(\tilde{A})$, where $\tilde{A}$ is the $p$-adic completion of $\Zp[x,x^{-1}]$. A crystal on $X$ corresponds to a finitely generated $\tilde{A}$-module $M$ with a connection $\nabla: M \to M \otimes \tilde{A}\, dx$ such that $\nabla$ is topologically quasi-nilpotent. For example, for $n$ prime to $p$, we can take $\nabla\colon \tilde{A}\to \tilde{A}\,dx$  to be the connection such that $\nabla(1)\coloneqq -\frac{1}{n}\frac{dx}{x}$. This connection is topologically quasi-nilpotent because $$\nabla_{\partial_x}\colon \F_p[x,x^{-1}]\to \F_p[x,x^{-1}] $$ is a nilpotent $\F_p[x^p,x^{-p}]$-linear endomorphism of index $p$ (check it on the basis $1,x^{-1},\dots,x^{-p+1}$).
	\end{exam}
\begin{rema}(Finite $p$-basis)		Let $A$ be a reduced $\F_p$-algebra. A \textit{$p$-basis} of $A$ is a subset $\{x_i\}_{i \in I} \subseteq A$ such that the monomials $\prod_{i \in I} x_i^{e_i}$ with $0 \leq e_i \leq p-1$ form a basis of $A$ as an $A^p$-module. Such rings admit flat $p$-complete lifts to $\Zp$. If $A$ is a smooth algebra over a perfect field, then Zariski locally it admits a finite $p$-basis. A ring like $\Fp[[x]]$ also admits a finite $p$-basis even if it is not of finite type over $\Fp$. Theorem \ref{Thm:CrystalsAsConnections} still holds for these rings.
	\end{rema}

\section{Crystalline Dieudonné modules and overconvergence}

\subsection{\texorpdfstring{$p$}{p}-divisible groups}
It turns out that by replacing the $p^n$-torsion of an abelian variety with the inductive system of $p^n$-torsion for all $n$, one gets a much better behaved object. This leads to the notion of \textit{$p$-divisible groups}, which is the main object of study in this section.

\begin{defi}
    Let $X$ be a scheme and $p$ be a prime. A \textit{$p$-divisible group} (also called a \textit{Barsotti--Tate group}) of height $h$ over $X$ is a sheaf of abelian groups $\mathcal{G}$ on the big fppf site of $X$ such that the following conditions hold.
\begin{enumerate}
    \item $\mathcal{G} = \varinjlim_n \mathcal{G}[p^n]$. 
    \item $\mathcal{G} \xrightarrow{\cdot p} \mathcal{G}$ is surjective.
    \item $\mathcal{G}[p^n]$ is finite and locally free\footnote{It is enough to assume this condition when $n=1$.} over $X$ of order $p^{nh}$ for every $n\geq 1$.
\end{enumerate}
If $\mathcal{G}$ is a $p$-divisible group, its \textit{Cartier dual} $\mathcal{G}^\vee$ is defined as $\varinjlim_n \mathcal{G}[p^n]^\vee$, where the transition maps are given by the dual of the multiplication-by-$p$ map. We also denote by $\omega_{\mathcal{G}/X}$ the limit $\varprojlim_n e^*\Omega^1_{\mathcal{G}[p^n]/X}$.
\end{defi}

\begin{exam}
Here are the main examples of $p$-divisible groups.
\spa
\begin{itemize}
    \item[--] The \textit{constant} $p$-divisible group $\mathbb{Q}_p/\mathbb{Z}_p = \varinjlim_n \mathbb{Z}/p^n\mathbb{Z}$.
    
    \spa
    \item[--] The \textit{multiplicative} $p$-divisible group $\mu_{p^\infty} = \varinjlim_n \mu_{p^n}$.
    
    \spa
    \item[--] If $A$ is an abelian scheme over $X$, then $A[p^\infty] \coloneqq \varinjlim_n A[p^n]$ is a $p$-divisible group.
\end{itemize}

The Cartier dual of $\mathbb{Q}_p/\mathbb{Z}_p$ is $\mu_{p^\infty}$, and the Cartier dual of $A[p^\infty]$ is $A^\vee[p^\infty]$, where $A^\vee$ is the dual abelian scheme. 
\end{exam}

In characteristic $p$, the inductive system of non-reduced finite group schemes
\[
\mu_{p^n,\mathbb{F}_p}=\operatorname{Spec}(\mathbb{F}_p[t]/(t^{p^n}-1))
\]
is represented, on nilpotent test algebras, by the smooth formal multiplicative group
\[
\widehat{\mathbb{G}}_{m,\mathbb{F}_p}=\operatorname{Spf}(\mathbb{F}_p[[t-1]]).
\]
In this precise sense, passing from the individual $p^n$-torsion group schemes to the $p$-divisible group $\mu_{p^\infty}$ packages the entire system into a formal group. This is one of the main reasons why $p$-divisible groups are better behaved than finite flat group schemes of $p$-power order. Another main motivation to introduce $p$-divisible groups is that they control the deformation theory of abelian schemes. More precisely, we have the following result.

\begin{theo}[Serre--Tate]
Let $X$ be a scheme on which $p$ is locally nilpotent, and let $X_0 \hookrightarrow X$ be a closed immersion defined by a nilpotent ideal sheaf. Let $\calA_0$ be an abelian scheme over $X_0$. Then the natural functor
$$\left\{\text{Lifts of }\calA_0\text{ to }X\right\} \;\to\; \left\{\text{Lifts of }\calA_0[p^\infty]\text{ to }X\right\}$$
induces a canonical equivalence of categories.
\end{theo}

\begin{coro}
For a perfect field $k$ of characteristic $p$, the formal deformation functor of an abelian variety $A_0/k$ over $W$ is naturally isomorphic to the formal deformation functor of its $p$-divisible group $A_0[p^\infty]$ over $W$.
\end{coro}

By extending Dieudonné theory to a relative setting using crystals, we can give a more explicit description of the deformation functor of $p$-divisible groups and thus of abelian varieties. 

\begin{defi}
    Let $X$ be a scheme over $\mathbb{F}_p$, and let $\varphi\colon X\to X$ be the absolute Frobenius. A \textit{(crystalline) Dieudonné module} over $X$ is a triple $(\mathcal{M}, F, V)$ where $\mathcal{M}$ is a locally free crystal over $X$, and $F: \varphi^*\mathcal{M} \to \mathcal{M}$ and $V: \mathcal{M} \to \varphi^*\mathcal{M}$ are maps such that $FV = VF = p$. We denote by $\mathrm{DM}(X)$ the category of Dieudonné modules over $X$. There is a similar notion of a dual Dieudonné module $\mathcal{M}^\vee$ with underlying crystal the internal hom $\underline{\mathrm{Hom}}(\mathcal{M},\mathcal{O}_\mathrm{cris})$. If $c_{\mathcal M}\colon\varphi^*(\mathcal M^\vee)\xrightarrow{\sim}(\varphi^*\mathcal M)^\vee$ denotes the canonical isomorphism, then $F_{\mathcal M^\vee}\coloneqq V^\vee\circ c_{\mathcal M}$ and $V_{\mathcal M^\vee}\coloneqq c_{\mathcal M}^{-1}\circ F^\vee$.
\end{defi}
Theorem \ref{Thm:DieudonneOda} extends to a relative setting as follows.

\begin{theo}[Grothendieck, {\cite[Thm. V.1.6]{Mes72}}, \cite{BBM82}, \cite{deJ95}]\label{thm:crystallineDieudonne}
    Let $X$ be a smooth scheme over a perfect field of characteristic $p$. There is an exact equivalence 
    $$\mathbb{D}\colon \left\{\text{$p$-divisible groups over $X$} \right\}^\mathrm{op}\simeq \mathrm{DM}(X)$$ 
    satisfying the following properties.
    \begin{enumerate}
        \item $\mathcal{G}$ is of height $h$ if and only if $\mathbb{D}(\mathcal{G})$ is of rank $h$.
        
        \spa
        \item $\mathbb{D}(\mathcal{G}^\vee)=\mathbb{D}(\mathcal{G})^\vee.$
        
        \spa
        \item If $f\colon \calA\to X$ is an abelian scheme over $X$, there is a canonical isomorphism 
        $$\mathbb{D}(\calA[p^\infty])= R^1f_*\mathcal{O}_\mathrm{cris}.$$
        
        \spa
        \item For $(B\twoheadrightarrow A,\gamma)\in X_\mathrm{cris}$ and a lift $\mathcal{G}_B$ of $\mathcal{G}_A$ to $B$, there is a canonical \textit{Hodge filtration} 
        $$0\to \omega_{\mathcal{G}_B/B}\to \mathbb{D}(\mathcal{G})_B\to \omega^\vee_{\mathcal{G}_B^\vee/B}\to 0.$$
        
        \spa
        \item {For $(B\twoheadrightarrow A,\gamma)\in X_\mathrm{cris}$ whose divided powers are nilpotent}\footnote{{If $I\coloneqq \Ker(B\to A)$ and $I^{[n]}$ is generated by products $\gamma_{n_1}(x_1)\cdots\gamma_{n_r}(x_r)$ with $n_1+\cdots+n_r\geq n$, we ask that $I^{[n]}=0$ for $n\gg0$. This hypothesis cannot be omitted. At $p=2$, the canonical divided powers on $(2)\subseteq\mathbb Z/4\mathbb Z$ are not nilpotent, and classification by the lifted Hodge filtration fails for arbitrary $p$-divisible groups. Variants remain valid under connectedness hypotheses, see \cite[\S1.2.1]{VasiuST}. The unrestricted theory at $p=2$ is naturally formulated using Dieudonn\'e displays, see \cite{LauDisplays}.}}, the lifts of $\mathcal{G}_A$ to $B$ are in bijection with the lifts of the Hodge filtration of $\mathbb{D}(\mathcal{G})_A$ to $\mathbb{D}(\mathcal{G})_B$.
    \end{enumerate}
\end{theo}

\begin{coro}
    Let $\mathcal{G}$ be a $p$-divisible group over a perfect field $k$ of characteristic $p$ of height $h$ and dimension {$d\coloneqq\dim_k\omega_{\mathcal G/k}$}. Then the formal deformation functor $\mathrm{Def}(\mathcal{G}/W(k))$ is representable by a smooth formal scheme over $W(k)$ of dimension $d(h-d)$. In other words, $\mathrm{Def}(\mathcal{G}/W(k))$ is non-canonically isomorphic to $\operatorname{Spf}(W(k)[[x_1,\ldots,x_{d(h-d)}]])$.
\end{coro}
\begin{proof}
    {Apply Theorem \ref{thm:crystallineDieudonne}(5) successively to square-zero divided-power thickenings. The usual induction on the length of an Artinian base shows that deforming $\mathcal G$ has the same infinitesimal lifting theory as deforming its Hodge filtration. The latter is the formal deformation functor of the point defined by the Hodge filtration in $\mathrm{Gr}(d,h)_W$.} This is a smooth formal scheme over $W$ of dimension $d(h-d)$.
\end{proof}

\begin{exam}[Families of ordinary elliptic curves]\label{exam:FamiliesOrdinary}
    Let $X$ be a smooth connected scheme over a perfect field $k$ of characteristic $p$ and let $f\colon \calE\to X$ be a family of ordinary elliptic curves, meaning that for every closed point $x\in X$, the fiber $\calE_x$ is an ordinary elliptic curve. Then the $p$-divisible group $\mathcal{G}=\calE[p^\infty]$ sits in the exact sequence
    $$
    0 \to \mathcal{G}^\circ \to \mathcal{G} \to \mathcal{G}^{\text{\'et}} \to 0,
    $$
    where $\mathcal{G}^\circ$ is the connected component of $\mathcal{G}$ and $\mathcal{G}^{\text{\'et}}$ is the maximal \'etale quotient. Both $\mathcal{G}^\circ$ and $\mathcal{G}^{\text{\'et}}$ are $p$-divisible groups of height $1$. The former is isomorphic to $\mu_{p^\infty}$ when base changed to every $\bar{k}$-point of $X$, while the latter is pointwise isomorphic to $\mathbb{Q}_p/\mathbb{Z}_p$. The Dieudonné module $\mathbb{D}(\mathcal{G})$ can be then written as an extension 
    $$
    0 \to \mathbb{D}(\mathcal{G}^{\text{\'et}}) \to \mathbb{D}(\mathcal{G}) \to \mathbb{D}(\mathcal{G}^\circ) \to 0,
    $$
    which is a $p$-adic lift of the conjugate filtration on $H^1_{\mathrm{dR}}(\calE/X)$ when $X$ is affine. {If the Serre--Tate deformation parameters of the family are non-constant, then this extension is non-trivial as an extension of Dieudonné crystals. Indeed, an $F$- and $V$-compatible splitting would, by Theorem~\ref{thm:crystallineDieudonne}, split the connected--\'etale extension of the $p$-divisible group, whereas the Serre--Tate parameter measures precisely this extension class}\footnote{{An abelian scheme $\calA\to X$ is \emph{isotrivial} if, after a finite \'etale cover $X'\to X$, its pullback becomes constant.}}.
\end{exam}

Example \ref{exam:FamiliesOrdinary} highlights a difference with the $\ell$-adic world, where $R^1f_*\mathbb{Q}_\ell$ for a non-isotrivial family of ordinary elliptic curves is irreducible. More generally, we have the following fundamental result on the structure of $\ell$-adic cohomology of smooth proper morphisms over $\overline{\mathbb{F}}_p$.

\begin{theo}[Deligne's semi-simplicity theorem, {\cite[Cor. 3.4.13]{Del80}}]\label{thm:Semisimplicity}
    For any smooth proper morphism $f\colon Y \to X$ of smooth schemes over $\overline{\mathbb{F}}_p$, the lisse $\mathbb{Q}_\ell$-sheaves $R^if_*\mathbb{Q}_\ell$ are semisimple.
\end{theo} 
Thanks to the parabolicity theorem, we now have a way to determine precisely the relation between the behavior of $\mathbb{D}(\mathcal{G})$ and $R^1f_*\mathbb{Q}_\ell$ for {an abelian scheme over $X$ whose Newton polygon is constant, using also Theorem~\ref{thm:Independence}.} To state the theorem we need to introduce the notion of the \textit{slope filtration} of an \textit{$F$-isocrystal}, \textit{overconvergent $F$-isocrystal}, and \textit{monodromy groups}.

\subsection{\texorpdfstring{$F$}{F}-isocrystals and the slope filtration}

In the $\ell$-adic setting, the category of lisse $\mathbb{Q}_\ell$-sheaves is a monoidal rigid $\mathbb{Q}_\ell$-linear category. When the base is connected, it also admits a fibre functor given by the stalk at a geometric point. Thus it is a \textit{neutral Tannakian category}.

\spa

The category of crystals over any base is a monoidal category, and the Frobenius endofunctor $F^*$ is monoidal. Nonetheless, the category of Dieudonné modules is not a monoidal category as the structure maps $F$ and $V$ are not naturally defined on a tensor product in a way that preserves the condition $FV = p$. However, if we forget the $V$-structure, we get a monoidal category of \textit{$F$-crystals}, which are crystals equipped with a Frobenius structure $\Phi\colon F^*\mathcal{M}\to \mathcal{M}$ that becomes an isomorphism after inverting $p$. In order to have an abelian category and to have a notion of duals, we need to invert $p$. This leads to the following definition.

\begin{defi}
    Let $X$ be a scheme over $\mathbb{F}_p$. We write $\mathrm{Isoc}(X)$ for the category of finitely presented crystals up to isogeny. An \textit{$F$-isocrystal} over $X$ is an isocrystal $\mathcal{M}$ equipped with an isomorphism $\Phi\colon F^*\mathcal{M} \xrightarrow{\sim} \mathcal{M}$. We denote by $\Fisoc(X)$ the category of $F$-isocrystals over $X$.
\end{defi}

\begin{prop}
    If $X$ is a smooth variety over a perfect field $k$ of characteristic $p$, then the categories $\mathrm{Isoc}(X)$ and $\Fisoc(X)$ are rigid monoidal abelian categories. In addition, $\mathrm{Isoc}(X)$ is $K$-linear and $\Fisoc(X)$ is a $\Qp$-linear category. If $X$ is {geometrically connected}, then $\mathrm{Isoc}(X)$ and $\Fisoc(X)$ are Tannakian categories over $K$ and $\Qp$, respectively. 
\end{prop}
\begin{rema}
	The evaluation at a perfect point $x$ gives a fibre functor for $\mathrm{Isoc}(X)$ after extending scalars to $K_x=W(\kappa(x))[1/p]$. The category $\Fisoc(X)$, instead is in general non-neutral over $\Qp$.
\end{rema}

When $k$ is a perfect field, the category $\mathrm{Isoc}(k)$ is equivalent to the category of finite-dimensional $K$-vector spaces. Consequently, an object of $\Fisoc(k)$ corresponds to a finite-dimensional $K$-vector space $V$ equipped with an isomorphism $\Phi\colon \sigma^*V \xrightarrow{\sim} V$, where $\sigma$ is the Frobenius automorphism of $K$. If $k$ is also assumed algebraically closed, one has the following fundamental diagonalisation result for $F$-isocrystals.

\begin{theo}[Dieudonné--Manin]
    If $k$ is algebraically closed, then every $F$-isocrystal $(M, \Phi)$ over $\operatorname{Spec}(k)$ decomposes as 
    $$M = \bigoplus_{\alpha\in \mathbb{Q}} M_{\alpha}$$ 
    where $M_{\alpha}$ denotes the $K$-linear subspace generated by vectors $v\in M$ such that $\Phi^r(v) = p^s v$ for $r>0$, $s\in \Z$, and $s/r=\alpha$. In addition, $M_\alpha$ decomposes as a direct sum of $F$-isocrystals isomorphic to $K\{F\}/(F^r-p^s)$ with Frobenius given by $F\cdot a=\sigma(a)\cdot F$, {where $r>0$ and $s$ are chosen with $\gcd(r,s)=1$.} 
\end{theo}

\begin{defi}[Slope]
    For an algebraically closed field $k$ of characteristic $p$, the set of \textit{slopes} of an $F$-isocrystal $(M, \Phi)$ over $k$ is the set of rational numbers $\alpha$ such that $M_\alpha\neq 0$. The \textit{multiplicity} of a slope $\alpha$ is defined as the dimension of $M_\alpha$. For a general field $k$, the slopes of an $F$-isocrystal are defined as the slopes of its base change to $\bar{k}$.
\end{defi}

It is useful to present the slopes of an $F$-isocrystal in a graphical way using the Newton polygon.

\begin{defi}[Newton polygon]\label{def:NewtonPolygon}
    Let $(M,\Phi)$ be an $F$-isocrystal over a field $k$ of characteristic $p$. Let $\alpha_1<\alpha_2<\dots<\alpha_m$ be the distinct slopes and let $d_1,\dots, d_m$ be the multiplicities of these slopes. The \emph{Newton polygon} $\mathrm{NP}(M,\Phi)$ is the convex polygonal line in $\mathbb{R}^2$ starting at $(0,0)$ whose consecutive segments have slopes $\alpha_1,\dots,\alpha_m$ and horizontal lengths $d_1,\dots,d_m$. Equivalently, its vertices are the points
    $$
    P_i=\left(\sum_{j=1}^{i} d_j,\;\sum_{j=1}^{i} \alpha_j d_j\right),\qquad i=0,1,\dots,m.
    $$
\end{defi}

\begin{exam}\label{exam:NPEllipticCurves}
    For an elliptic curve $E$ over an algebraically closed field $k$ of characteristic $p$, let $(M,\Phi)$ be the rationalised Dieudonné module $\mathbb{D}(E[p^\infty])[1/p]$. The example discussed in Example \ref{exam:Dieudonne} lifts in the following way.
    
    \begin{itemize}
        \item[--] If $E$ is ordinary, then $M=K\oplus K$ with 
        $$\Phi=\begin{pmatrix} \sigma & 0 \\ 0 & p\sigma \end{pmatrix}.$$ 
        In this case, the slopes of $(M,\Phi)$ are $0$ and $1$.
        
        \spa
        \item[--] If $E$ is supersingular, then $M$ is again of dimension $2$ and there exists a basis such that 
        $$\Phi=\begin{pmatrix} 0 & p\sigma \\ \sigma & 0 \end{pmatrix}.$$ 
        In this case, the slope of $(M,\Phi)$ is $\frac{1}{2}$ with multiplicity $2$. In contrast to the Dieudonné module of $E[p]$, the $F$-isocrystal $(M,\Phi)$ is irreducible.
    \end{itemize} 

    Here are the two possibilities for the Newton polygons.
    $$
    \begin{array}{c@{\hspace{2cm}}c}
    \text{ordinary} & \text{supersingular} \\[1ex]
    \begin{tikzpicture}[scale=0.8]
    \draw[->] (0,0) -- (2.5,0) node[right]{};
    \draw[->] (0,0) -- (0,2) node[above]{};
    \draw[thick] (0,0) -- (1,0) -- (2,1);
    \fill (0,0) circle (2pt) (1,0) circle (2pt) (2,1) circle (2pt);
    \end{tikzpicture}
    &
    \begin{tikzpicture}[scale=0.8]
    \draw[->] (0,0) -- (2.5,0) node[right]{};
    \draw[->] (0,0) -- (0,2) node[above]{};
    \draw[thick] (0,0) -- (2,1);
    \fill (0,0) circle (2pt) (2,1) circle (2pt);
    \end{tikzpicture}
    \end{array}
    $$
    
    When $\calE\to X$ is a family of elliptic curves with connected basis and ordinary generic fibre, there exists a maximal closed subscheme $Z\subseteq X$ where the fibres are supersingular. This is called the \textit{supersingular locus}. The complement of $Z$, denoted by $U$, is the \textit{ordinary locus}. The Newton polygon of $\mathbb{D}(\calE[p^\infty])[1/p]$ is constant on $U$, but it jumps at the points of $Z$. In particular, the Newton polygon is not locally constant in general. The filtration on $\mathbb{D}(\calE[p^\infty])$ we discussed in Example \ref{exam:FamiliesOrdinary} does not extend to the whole $X$. This is a special case of the following fundamental results.
\end{exam}
\begin{theo}[Grothendieck, {\cite{KatzSlope,deJongOortPurity}}]\label{thm:NewtonPolygonSemicontinous}
	Let $X$ be a smooth variety over a perfect field of characteristic $p$, and let $(\mathcal{M}, \Phi)$ be an $F$-isocrystal over $X$. The Newton polygon of $(\mathcal{M}, \Phi)$ is an upper semicontinuous function for the Zariski topology of $X$ with locally constant endpoints. In addition, the locus where the Newton polygon {differs from its value at the generic point}, if non-empty, has pure codimension $1$.
	
\end{theo}
In the loci where the Newton polygon is constant, we have the following structure theorem for $F$-isocrystals, which can be seen as a generalisation of the Dieudonné--Manin classification to a relative setting.
\begin{theo}[{\cite{KatzSlope},\cite{Ked16}, \cite[Cor. 2.6]{Tsuzuki}}]
    Let $X$ be a smooth variety over a perfect field, and let $(\mathcal{M}, \Phi)$ be an $F$-isocrystal over $X$ with constant Newton polygon. Then there exists a unique filtration
    $$
    0 = S_0(\mathcal{M}) \subseteq S_1(\mathcal{M}) \subseteq \cdots \subseteq S_m(\mathcal{M}) = \mathcal{M}
    $$
    by sub-$F$-isocrystals such that each graded piece $S_i(\mathcal{M})/S_{i-1}(\mathcal{M})$ has pure slope $\alpha_i$ for every {geometric} point of $X$ and $\alpha_1 < \alpha_2 < \cdots < \alpha_m$.
\end{theo}

\subsection{Newton stratification for \texorpdfstring{$\mathcal{A}_{g}$}{A-g}}
This section is an interlude on the geometry of the moduli space of abelian varieties in characteristic $p$ and its stratification by Newton polygons. This is a fundamental application of Theorem \ref{thm:NewtonPolygonSemicontinous}. It can be skipped at first reading.

\spa
We denote by $\mathcal{A}_{g}$ the coarse moduli space over $\overline{\mathbb{F}}_p$ of principally polarised abelian varieties of dimension $g$. For a point $x\in\mathcal{A}_{g}({\ovfp})$ corresponding to an abelian variety $A_x$, we consider its $p$-divisible group $A_x[p^\infty]$ and the associated Dieudonné module $\mathbb{D}(A_x[p^\infty])$. After inverting $p$, we obtain an $F$-isocrystal $(M_x,\Phi_x)$ over {$\ovfp$} of rank $2g$. The {Newton polygon} $\mathrm{NP}(A_x)$ of $A_x$ is defined as the Newton polygon of $(M_x,\Phi_x)$ as in Definition \ref{def:NewtonPolygon}. It is a convex polygon with breakpoints at integer coordinates, starting at $(0,0)$ and ending at $(2g,g)$. {Its slopes are symmetric under $\alpha\leftrightarrow1-\alpha$, because a polarisation identifies $A_x$ with its dual up to isogeny.}

\begin{defi}
    Let $\nu$ be a Newton polygon. The \emph{Newton stratum} associated to $\nu$ is the locally closed subvariety $\mathcal{N}_\nu\subseteq \mathcal{A}_{g}$ such that
    $$
    \mathcal{N}_\nu(\ovfp) = \{ x\in\mathcal{A}_{g}(\ovfp) \mid \mathrm{NP}(A_x) = \nu \}.
    $$ The existence of $\mathcal{N}_\nu$ is a consequence of Theorem \ref{thm:NewtonPolygonSemicontinous}.
\end{defi}

\begin{exam}[$g=1$]
    For elliptic curves, there are only two possible Newton polygons: the ordinary one with slopes $0$ and $1$, and the supersingular one with a single slope $1/2$ of multiplicity $2$. The ordinary locus is open and dense in $\mathcal{A}_{1}$, while the supersingular locus consists of a finite set of points.
\end{exam}

\begin{exam}[$g=2$]
    For principally polarised abelian surfaces, there are three possible Newton strata.
    \begin{itemize}
        \item[--] Ordinary: slopes $0,0,1,1$.
        \item[--] Almost ordinary: slopes $0,1/2,1/2,1$.
        \item[--] Supersingular: slopes $1/2,1/2,1/2,1/2$.
        
    \end{itemize}
    The ordinary locus is open and dense; the supersingular locus is closed of dimension $1$ and {is a union of rational curves meeting at \textit{superspecial} points}; the almost ordinary stratum has dimension $2$ and it is smooth and irreducible {\cite[Thm. A]{ChaiElladicmonodromy}}. \end{exam}

\subsection{Overconvergent \texorpdfstring{$F$}{F}-isocrystals}

While $F$-isocrystals should be the analogue of lisse $\ell$-adic sheaves, as we have seen they exhibit some different behaviours (see Example \ref{exam:FamiliesOrdinary} and the discussion after it). Berthelot introduced the notion of \emph{overconvergent $F$-isocrystals}, which have better cohomological properties and have closer ties with $\ell$-adic sheaves. These are defined using the language of rigid analytic geometry. A more algebraic definition of overconvergent $F$-isocrystals is obtained using a variant of the crystalline site called the \textit{marked crystalline site} \cite{Edged}.

\subsubsection{Overconvergent \texorpdfstring{$F$}{F}-isocrystals}
\label{sec:OCGm}

To present this concept, we start again with a concrete example.
Assume that $p>2$ and put
\[
X\coloneqq\mathbb{P}^1_k\setminus\{0,1,\infty\}.
\]
Consider the Legendre family
$f\colon\mathcal{E}\to X$. The scheme $X$ has an evident smooth lift
\[
X_W\coloneqq\mathbb{P}^1_W\setminus\{0,1,\infty\}.
\]
{Indeed, the three boundary points lift to the pairwise disjoint sections $0$, $1$, and $\infty$ of $\mathbb P^1_W$, and removing these sections commutes with reduction modulo $p$.}
Let $\widehat{X}_W$ denote its $p$-adic completion. Its rigid generic
fibre is
\[
\mathfrak{X}
\coloneqq 
\mathbb{P}^{1,\mathrm{an}}_K
 \setminus
 \bigl(
   \mathring{D}(0,1)
   \cup\mathring{D}(1,1)
   \cup\mathring{D}(\infty,1)
 \bigr).
\]

By Theorem \ref{Thm:CrystalsAsConnections}, the $F$-isocrystal $\mathcal{M}
\coloneqq
\mathbb{D}(\mathcal{E}[p^\infty])[1/p]$
defines a connection with Frobenius on $\mathfrak{X}$.

\spa

Let $Z\subseteq X$ be the supersingular locus of the Legendre
family. It is a nonempty finite set, namely the zero locus of the
Hasse invariant of the family. After replacing $k$ by a finite
extension, we may write
\[
Z=\{a_1,\ldots,a_r\}\subseteq X(k).
\]
{Choose lifts $\tilde{a}_i\in W$ of the points $a_i$.} Under the
specialization map, the point $a_i$ corresponds to the open
$p$-adic unit disk
\[
]a_i[
=
\mathring{D}(\tilde{a}_i,1)
=
\{t\in\mathfrak{X}\mid |t-\tilde{a}_i|<1\}.
\]
Consequently, the tube of the supersingular locus is the disjoint
union
\[
\mathfrak{Z}
\coloneqq ]Z[
=
\coprod_{i=1}^r
\mathring{D}(\tilde{a}_i,1).
\]

Let
\[
U\coloneqq X\setminus Z
\]
be the ordinary locus. Its tube is
\[
\mathfrak{U}
\coloneqq ]U[
=
\mathfrak{X}\setminus\mathfrak{Z}.
\]
Thus, in the rigid-analytic space $\mathfrak{X}$, passing from $X$
to its ordinary locus amounts to removing the open unit disks
corresponding to the supersingular points.

\spa

For $0<\rho<1$, with $\rho$ sufficiently close to $1$, consider
\[
\mathfrak{U}_\rho
\coloneqq
\mathfrak{X}\setminus
\coprod_{i=1}^r
\mathring{D}(\tilde{a}_i,\rho).
\]
We then have
\[
\mathfrak{U}\subseteq\mathfrak{U}_\rho\subseteq\mathfrak{X}.
\]
The space $\mathfrak{U}_\rho$ is obtained from $\mathfrak{U}$ by
adjoining, inside each supersingular disk, an annular region
\[
\mathfrak{A}_{i}(\rho,1)
\coloneqq
\{t\mid \rho\leq |t-\tilde{a}_i|<1\}.
\]

Over the ordinary locus $U$, the $F$-isocrystal $\mathcal{M}$
admits a slope $0$ subobject
\[
\mathcal{M}_0
\subseteq
\mathcal{M}|_U.
\]
It defines a connection with Frobenius on $\mathfrak{U}$. A
fundamental observation, due to Dwork, is that $\mathcal{M}_0$
does not extend, as a connection with Frobenius, to any of the
(strict) neighbourhoods $\mathfrak{U}_\rho$, no matter how close
$\rho$ is to $1$. In other words, the slope $0$ part cannot be
analytically continued across even an arbitrarily small annular
region inside the supersingular disks. Notice that the obstruction does not come from the full
$F$-isocrystal $\mathcal{M}$: the latter is defined across the
supersingular locus, on all of $\mathfrak{X}$. It is the slope
decomposition on the ordinary locus that fails to extend across
the points where the Newton polygon jumps. In addition, $\mathcal{M}$ extends to
sufficiently small annuli around $0$, $1$, and $\infty$. Thus
$\mathcal{M}$ has an extra $p$-adic extension property,
which the slope $0$ subobject does not.

\spa

According to Berthelot's definition, an \emph{overconvergent
$F$-isocrystal on $U$ relative to $X$} is represented by a flat
connection with Frobenius on some neighbourhood
$\mathfrak{U}_\rho$ of $\mathfrak{U}$ in $\mathfrak{X}$. Dwork's
observation therefore says that $\mathcal{M}_0$ is an
$F$-isocrystal on $U$ which is not overconvergent relative to
$X$. This phenomenon, which has no direct $\ell$-adic analogue,
is one of the main reasons for distinguishing overconvergent
$F$-isocrystals from arbitrary  $F$-isocrystals.

\spa

In general, for any open immersion $U\subseteq Y$ of varieties,
Berthelot constructed a category of overconvergent
$F$-isocrystals on $U$ relative to $Y$. We denote by $\Foi(X)$
the category of overconvergent $F$-isocrystals with respect to a
compactification $X\hookrightarrow\overline{X}$.
Berthelot proved that this category is independent of the chosen
compactification {\cite{Ber96}}. There is a natural restriction functor
\[
\Foi(X)\hookrightarrow\Fisoc(X),
\]
which is fully faithful {when $X$ is smooth} by \cite{KedFull}. An $F$-isocrystal in
the essential image of this functor is called
\emph{$\dagger$-extendable}. {If $X$ is separated and of finite type,
every overconvergent $F$-isocrystal $\calE$ on $X$ has finite-dimensional
rigid cohomology, both with and without compact supports \cite{Ked06a}.}
Overconvergent $F$-isocrystals satisfy a Lefschetz trace formula {\cite{EtesseLeStum}} and are the correct objects for
the Langlands correspondence \cite{Abe18}.

\subsection{Monodromy groups}
Lisse $\Ql$-sheaves can be described as representations of the \'etale fundamental group. The image of such representation measures the ``size'' of the variation of the sheaf. For $F$-isocrystals we do not have such a representation, but we can use the \textit{Tannakian formalism}. This helps to understand the relation between $F$-isocrystals and overconvergent $F$-isocrystals. 
\begin{defi}
    Let $X$ be as above, and let $\mathcal{M} \in \mathrm{Isoc}(X)$. Denote by $\langle \mathcal{M} \rangle$ the Tannakian subcategory generated by $\mathcal{M}$ (the smallest full subcategory containing $\mathcal{M}$ that is closed under subquotients, direct sums, tensor products, and duals). For a fiber functor coming from a perfect point $x$, put {$K_x=W(\kappa(x))[\tfrac{1}{p}]$}. The \emph{monodromy group} $G(\mathcal{M}, x)$ is the algebraic group scheme over {$K_x$} of tensor automorphisms of the restriction of the fibre functor to $\langle \mathcal{M} \rangle$. By Tannaka duality, we have that $$\langle \mathcal{M} \rangle {\otimes_K K_x}= \mathrm{Rep}_{K_x}(G(\mathcal{M}, x)).$$ The same definition applies to overconvergent isocrystals, giving the monodromy group $G(\mathcal{M}^\dagger, x)$.  If $\mathcal{M}$ is constant (it comes from the base field), then $G(\mathcal{M}, x)$ is trivial. If $\mathcal{M}$ varies non-trivially, then $G(\mathcal{M}, x)$ can be large.

\end{defi}

One of the motivations for studying monodromy groups of overconvergent $F$-isocrystals is the following independence-of-$\ell$ phenomenon.

\begin{theo}[\cite{Monodromy}]\label{thm:Independence}
    Let $f\colon Y\to X$ be a smooth proper morphism of smooth varieties over $\overline{\mathbb{F}}_p$, {let $\ell\ne p$, and assume that $X$ is geometrically connected}. For every $x\in X(\overline{\mathbb{F}}_p)$ and $i\geq 0$ and after choosing an isomorphism $\Qlbar\simeq \overline{K}$, the groups
    $$G(R^if_*\mathbb{Q}_\ell,x)^\circ_{\overline{\mathbb{Q}}_\ell}\quad\text{and}\quad G(R^if_*\mathcal{O}^\dagger,x)_{\overline{K}}^\circ$$
    {are non-canonically isomorphic.}
\end{theo}

This theorem allows us to transfer knowledge about $\ell$-adic monodromy to the $p$-adic setting. In the next lecture, we will see how the parabolicity conjecture can be combined with Theorem \ref{thm:Independence} to determine the monodromy groups of $\mathbb{D}(\calA[p^\infty])$ for any abelian scheme $\calA\to X$ with constant Newton polygon.

\section{Parabolicity conjecture and applications}
\label{Sec:ParabolicityConjecture}
\subsection{The statement}

Let $X$ be a smooth {geometrically connected} variety over a perfect field $k$, and let $(\calM^\dagger,\Phi)$ be an overconvergent $F$-isocrystal over $X$. Denote by $(\calM,\Phi)$ the associated $F$-isocrystal (the image under the restriction functor). Assume that $(\calM,\Phi)$ has constant Newton polygon, so it admits a slope filtration $S_\bullet(\calM)$. Fix a perfect point $x$ of $X$. Recall that we have two monodromy groups: the overconvergent monodromy group $G(\calM^\dagger, x)$ and the monodromy group $G(\calM, x)$. 
There is a natural inclusion $$G(\calM, x) \subseteq  G(\calM^\dagger, x).$$

\begin{theo}[Parabolicity conjecture, {\cite{Parabolicity}}]\label{thm:Parabolicity}
		The inclusion $G(\calM, x) \subseteq  G(\calM^\dagger, x)$ identifies $G(\calM, x)$ with the stabiliser of the slope filtration of $\calM_{x}$. Moreover, if $\calM^\dagger$ is semi-simple, then $G(\calM, x)$ is a parabolic subgroup of $G(\calM^\dagger, x)$.
	\end{theo}
	
	\begin{exam}
		In the situation of Example \ref{exam:FamiliesOrdinary}, {assuming that the family is non-isotrivial and has full geometric monodromy}, if $\calM\coloneqq \mathbb{D}(\calE[p^\infty])[\tfrac{1}{p}]=R^1f_*\mathcal{O}_{\mathrm{cris}}[\tfrac{1}{p}]$ and $\calM^\dagger\coloneqq R^1f_{*} \mathcal{O}^\dagger$, then the containment $$G(\calM,x)\subseteq G(\calM^\dagger,x)$$ corresponds to $$ B_2\subseteq \mathrm{SL}_2,$$ where $B_2$ is the Borel subgroup of $\mathrm{SL}_2$ consisting of upper triangular matrices. The tautological representation of $\mathrm{SL}_2$ is irreducible, while $B_2$ admits a non-trivial subrepresentation of rank $1$ which does not have any retraction. Thus the knowledge of the monodromy groups in particular recovers the fact that $\calM^\dagger$ is irreducible (see \S \ref{sec:OCGm}), while $\calM$ admits a non-trivial filtration. 
	\end{exam}

	Theorem \ref{thm:Parabolicity} should be seen as a ``big monodromy'' result for $F$-isocrystals. It can be used, for example, to prove the semi-simplicity of the graded pieces of the slope filtration, as we will see in the next section.
\subsubsection{Monodromy of the graded object}
		In the situation of Theorem \ref{thm:Parabolicity}, suppose that $\calM^\dagger$ is semi-simple. Then $G(\calM,x)\subseteq G(\calM^\dagger,x)$ is a parabolic subgroup. We can consider the Levi decomposition of 
		\begin{equation}\label{eq:Levi}
G(\calM,x)=U\rtimes L
		\end{equation} associated to the slope decomposition of $x$. In other words, $L$ is the subgroup of $G(\calM,x)$ that respects the slope decomposition of $\calM_x$, while $U$ is the unipotent radical of $G(\calM,x)$. If $\calN$ is the graded object associated to the slope filtration of $\calM$, then $L$ is the monodromy group of $\calN$. In particular, since $L$ is reductive, then $\calN$ is a semi-simple isocrystal. This proves the following corollary of Theorem \ref{thm:Parabolicity}, which is a $p$-adic analogue of Theorem \ref{thm:Semisimplicity}.
{The slope grading is a tensor grading of the fibre functor, so its cocharacter lies in $G(\calM^\dagger,x)$; this is what identifies $L$ as a Levi subgroup.}

		\begin{coro}
			Let $X$ be a smooth {geometrically connected} variety over a perfect field $k$, and let $(\calM^\dagger,\Phi)$ be a semi-simple overconvergent $F$-isocrystal over $X$ with constant Newton polygon. For every $i\geq 0$, the quotients $S_i(\calM)/S_{i-1}(\calM)$ of the slope filtration of $\calM$ are semi-simple $F$-isocrystals. {In particular, this applies to $R^1f_*\mathcal O^\dagger$ for an abelian scheme $f\colon Y\to X$ with constant Newton polygon by \cite[Theorem 5.1.6]{Parabolicity}.}
		\end{coro}
\subsection{\texorpdfstring{$\dagger$}{Dagger}-hulls and the slope filtration}
	We will not try to give a complete proof of Theorem \ref{thm:Parabolicity} in these notes, but we will explain in detail the role of \textit{$\dagger$-hulls} of sub-$F$-isocrystals.
\begin{defi}Let $(\calM,\Phi)$ be a $\dagger$-extendable $F$-isocrystal over $X$ and let $(\mathcal{N},\Phi)$ be a sub-$F$-isocrystal of $(\calM,\Phi)$. The \textit{$\dagger$-hull} of  $(\mathcal{N},\Phi)\subseteq (\calM,\Phi)$ is the smallest $\dagger$-extendable subobject of $(\mathcal{M},\Phi)$ containing $(\mathcal{N},\Phi)$. It is denoted by $(\overline{\mathcal{N}},\Phi)$.
\end{defi}

The reason for this notion is the following key result.

\begin{theo}[{{\cite[Proposition 6.1]{Tsu19}} for curves, {\cite[Theorem 4.1.3]{Parabolicity}} in general}]\label{theo:MS}
Let $(\mathcal{N},\Phi) \subseteq (\mathcal{M},\Phi)$ be an inclusion of $F$-isocrystals over $X$. If $(\mathcal{M},\Phi)$ {is $\dagger$-extendable and} has constant Newton polygon, then $S_1(\mathcal{N}) = S_1(\overline{\mathcal{N}})$.
\end{theo}

Assuming Theorem \ref{theo:MS}, we can give a sketch of the proof of Theorem \ref{thm:Parabolicity} when $k$ is a finite field.

\subsubsection{Proof of Theorem \ref{thm:Parabolicity} over finite fields assuming Theorem \ref{theo:MS}}
Suppose that $k=\mathbb F_{p^n}$. After a finite extension of $k$, we may choose a rational point $x\in X(k)$; we also replace the Frobenius structure by its $n$-th power. Evaluation at $x$ on the underlying isocrystal then gives a $\mathbb Q_{p^n}$-linear fibre functor on the category of $F^n$-isocrystals generated by $(\calM,\Phi)$, and similarly in the overconvergent category. We can therefore take the Tannakian duals while keeping track of the Frobenius structure. We denote by $G$ and $G^\dag$ the two arithmetic monodromy groups we obtain. They are arithmetic variants of $G(\calM,x)$ and $G(\calM^\dag,x)$ and are better behaved for this problem because they also contain information about the slopes. The aim is to prove that the subgroup $G\subseteq G^\dagger$ is the stabiliser of the slope filtration. Since the slope filtration of $\calM_x$ extends to $(\calM,\Phi)$, we already know that $G$ is contained in $\mathrm{Stab}(S_\bullet)\subseteq G^\dagger$. It remains to show the opposite inclusion.

\spa

By Chevalley’s theorem, there exists an object $(\mathcal{N}^\dagger,\Phi)$ in the Tannakian category generated by $(\mathcal{M}^\dagger,\Phi)$ and a rank 1 subobject $(\mathcal{L},\Phi) \subseteq (\mathcal{N},\Phi)$ such that $G$ is the stabiliser of the line $\mathcal{L}_x \subseteq \mathcal{N}_x$. Since $(\mathcal{N},\Phi)$ is $\dag$-extendable, we can take the $\dagger$-hull of $(\mathcal{L},\Phi)$ in $(\mathcal{N},\Phi)$, denoted by $(\overline{\mathcal{L}},\Phi)$.  Its fibre $
\overline{\mathcal{L}}_x \subseteq \mathcal{N}_x
$
is a subspace containing $\mathcal{L}_x$.

\spa

The $F$-isocrystal $(\mathcal{L},\Phi)$ is of rank 1, thus it admits a unique slope; call it $s$.  On the fibre $\calN_x$ we have the slope filtration (it exists because $(\mathcal{N},\Phi)$ lies in the Tannakian category generated by $(\mathcal{M},\Phi)$).  Write $(\calN_x)^{\le s}$ for the subspace consisting of all vectors whose slope is $\le s$. This is one of the pieces of the slope filtration.

\spa

Since $S_1(\mathcal{L}) = \mathcal{L}$, Theorem \ref{theo:MS} tells us that $(\overline{\mathcal{L}},\Phi)$ has no subobjects of slope $< s$, and its part of slope  $s$ is precisely $\mathcal{L}$.  Passing to fibres, this yields the equality
$$
\overline{\calL}_x \cap (\mathcal{N}_x)^{\le s} = \calL_x.
$$

On the one hand, $\mathrm{Stab}(S_\bullet)$ stabilises $\overline{\calL}_x$ because $\overline{\calL}$ is $\dagger$-extendable, thus by Tannakian duality it is stabilised by $G^\dagger$. On the other hand, {the slope filtration is a tensor filtration of the fibre functor. Hence every element of $G^\dagger$ stabilising $S_\bullet(\calM_x)$ also stabilises the induced piece $(\mathcal{N}_x)^{\le s}$.} We deduce that $\mathrm{Stab}(S_\bullet)$ stabilises  $\calL_x$.  But $G$ is precisely the stabiliser of $\calL_x$, thus $\mathrm{Stab}(S_\bullet)$ is contained in $G$, as we wanted.

\subsection{The local enhancement}
	
A refined version of the parabolicity conjecture deals with the local monodromy at a closed point. Let $x \in X$ be a closed point, and let $X^{/x}$ denote the affine scheme $$\spec\left(\varprojlim_n\calO_{X,x}/\mathfrak{m}^n\right).$$ The scheme $X^{/x}$ is an irreducible \textit{Frobenius-smooth}\footnote{{An $\Fp$-scheme is \textit{Frobenius-smooth} if its absolute Frobenius is syntomic or, equivalently, if  it is Zariski locally the spectrum of a ring with a finite $p$-basis.}} Noetherian scheme and $\mathrm{Isoc}(X^{/x})$ is again a Tannakian category \cite[Cor. 3.3.3]{DrinfeldCrystals}. Thus, restricting an isocrystal $\calM$ to $X^{/x}$ produces a local monodromy group $G(\calM^{/x}, x)$. With van Hoften we proved the following fundamental result for this local monodromy group.
	
\begin{theo}[{\cite{DvH22}}]\label{thm:LocalEnhancement}
		{Assume that $X$ is irreducible and that $x\in X(k)$ is a rational point.} Let $(\calM^\dagger,\Phi)$ be a semi-simple overconvergent $F$-isocrystal with constant Newton polygon. Then the local monodromy group $G(\calM^{/x}, x)$ is equal to the unipotent radical of the global monodromy group $G(\calM, x)$.
	\end{theo}
	
	This might be surprising because in the $\ell$-adic setting, if the base field is algebraically closed, the local monodromy group at a point is trivial, since the étale fundamental group of a formal disk is trivial. In the $p$-adic setting, however, the local monodromy can be very large! This is a reflection of the fact that crystals have interesting deformations, as we have seen in the discussion about the deformation theory of $p$-divisible groups. Theorem \ref{thm:LocalEnhancement}, which is proved using Theorem \ref{thm:Parabolicity} can be thought of as a generalisation of the parabolicity conjecture. Indeed, it says that the unipotent group appearing in the Levi decomposition \eqref{eq:Levi} is realised as the local monodromy group at a closed point. Theorem \ref{thm:Parabolicity} and Theorem \ref{thm:LocalEnhancement} are some of the key ingredients in the proof of the \textit{Hecke orbit conjecture}.
\subsection{The Hecke orbit conjecture} 

Let $g$ be a positive integer, and $\mathcal{A}_g$ the moduli space of principally polarised abelian varieties of dimension $g$ over $\overline{\F}_p$. For a point $x=(A_x,\lambda)\in\mathcal{A}_g(\overline{\F}_p)$, the \textit{Hecke orbit} $\mathcal{H}(x)$ is the set of points $y\in\mathcal{A}_g(\overline{\F}_p)$ such that there exists an isogeny $A_x\to A_y$ compatible with polarisations up to a scalar. 
		
\begin{theo}[{\cite{ChaiOortBook}}, {\cite{DvH22}}]
		\label{Thm:HeckeOrbit}
		For each $x\in \mathcal{A}_g(\overline{\F}_p),$ the Hecke orbit $\mathcal{H}(x)$ is dense in the Newton stratum containing $x$. The same holds for Shimura subvarieties (e.g., subvarieties defined by an extra constraint on the endomorphisms of $(A_x,\lambda)$) of $\mathcal{A}_g$ at odd primes of good reduction.
	\end{theo}

For this question one notices immediately that there is an asymmetry between $p$-primary isogenies and prime-to-$p$ isogenies. Since the group schemes $A_x[p^n]$ are not étale, the $p$-primary isogenies can deform in families, while the prime-to-$p$ isogenies are rigid. For $\mathcal{A}_g$, one can reduce to the case of $\ell$-primary isogenies, and then the problem becomes to show that the Zariski closure is a \textit{central leaf}. The idea is to use the monodromy results to prove that this Zariski closure is big as well.

\subsubsection{Central leaves}

Within a given Newton stratum $\mathcal{N}_\nu$, the isomorphism class of the $p$-divisible group $A_x[p^\infty]$ can vary. On the other hand, $\ell$-primary isogenies preserve the isomorphism class of the $p$-divisible group. To capture this finer information, Chai and Oort introduced the notion of central leaves.

\begin{defi}[Central leaf]
Let $x \in \mathcal{A}_g(\overline{\mathbb{F}}_p)$ be a point corresponding to a principally polarised abelian variety $(A_x,\lambda_x)$. The \textit{central leaf} $\mathcal{C}(x)$ through $x$ is the closed subscheme of $\calN_\nu$ parametrising points {$y=(A_y,\lambda_y)$} in $\mathcal{A}_g(\overline{\mathbb{F}}_p)$ such that the polarised $p$-divisible groups $(A_x[p^\infty],\lambda_x)$ and $(A_y[p^\infty],\lambda_y)$ are isomorphic.
\end{defi}

Central leaves are smooth locally closed subvarieties of $\mathcal{A}_g$ and provide a foliation of each Newton stratum: the Newton stratum $\mathcal{N}_\nu$ is a disjoint union of central leaves {\cite{OortFoliation}}.

\subsection{Big monodromy of Hecke orbits}
Let $Z_\ell$ be the Zariski closure {inside the Newton stratum $\mathcal N_\nu$} of an $\ell$-primary Hecke orbit of a {non-supersingular} point $x \in \mathcal{A}_g(\overline{\mathbb{F}}_p)$. Since a universal abelian scheme does not exist over the coarse moduli space, we pass to a sufficiently fine prime-to-$p$ level cover. After this passage and after replacing $Z_\ell$ by a smooth connected dense open $Z_\ell^\circ$, let $f\colon \mathcal{A} \to Z_\ell^\circ$ be the universal abelian scheme. A deep result of Chai shows that the $\ell$-adic monodromy group of $R^1f_*\Ql$ over $Z_\ell^\circ$ is as large as possible, i.e. it is the full symplectic group $\mathrm{Sp}_{2g}$ {\cite[Prop. 4.1 and Prop. 4.4]{ChaiHeckeMonodromy}}. We transfer this information to the $p$-adic setting using all the results we have presented.

\spa

Let $\mathcal{M} \coloneqq \mathbb{D}(\mathcal{A}[p^\infty])[1/p]$ be the universal $F$-isocrystal and $\mathcal{M}^\dagger = R^1f_*\mathcal{O}^\dagger$ the universal overconvergent $F$-isocrystal over {$Z_\ell^\circ$}. By the independence-of-$\ell$ result (Theorem \ref{thm:Independence}), {$G(\mathcal{M}^\dagger,x)^\circ$ is abstractly isomorphic to $\mathrm{Sp}_{2g}$. The polarisation embeds it into $\mathrm{Sp}(\mathcal M_x,\lambda)\simeq\mathrm{Sp}_{2g}$; equality follows by connectedness and dimension.} The Newton polygon of the abelian scheme $\calA\to {Z_\ell^\circ}$ is constant, thus $\mathcal{M}$ admits a slope filtration and Theorem \ref{thm:Parabolicity} forces $G(\mathcal{M},x)$ to be the parabolic subgroup of $G(\mathcal{M}^\dagger,x)$ associated to the slope filtration. In addition, Theorem \ref{thm:LocalEnhancement} implies that the local monodromy {$G(\mathcal{M}^{/x},x)$} is the unipotent radical of that parabolic subgroup. The same argument shows that by replacing $Z_\ell$ with $\mathcal{C}(x)$, the local monodromy group of $\mathcal{M}$ over $\mathcal{C}(x)$ coincides with the one of $Z_\ell$.

\spa

The end of the proof of Theorem \ref{Thm:HeckeOrbit}, which is beyond the scope of these lectures, consists in translating the information about the local monodromy group into a statement about the dimension of the formal deformation space of $Z_\ell$ at $x$. This is done using a fundamental linearisation theorem by Chai and Oort, called the \textit{rigidity theorem}, which says that $Z_\ell$ can be described using a certain semi-linear algebra object: a \textit{Dieudonné-Lie algebra}  \cite{ChaiOortRigidity}. In order to relate Chai--Oort linearisation with our monodromy result, we use the theory of \textit{Cartier--Witt stacks,} developed by Drinfeld and Bhatt--Lurie \cite{BhattLurieII,DrinfeldPrismatization, DrinfeldCrystals}.

\subsection{Further reading}
For more details about the proofs of the theorems presented here, you can look at \cite{DA24} or the quoted articles. For further applications of the parabolicity conjecture and the Hecke orbit conjecture see for example \cite{PerfectPoints, BraggYang, J23, J25, LS25}.

	\begin{appendix}
\section{Isocrystals over formal schemes}
\label{App:IsocrystalsFormalSchemes}

\subsection{Setting}
Let $A$ be a Noetherian $\F_p$-algebra and $I\subseteq A$ an ideal such that $A$ is $I$-adically complete. Let $\mathfrak{X}$ be the formal scheme $\Spf(A)$, let $X$ be $\Spec(A)$, and let $X_n$ be $\Spec(A/I^{n})$ {for $n\geq1$}.  We write $\mathrm{Crys}(\mathfrak{X})$ for the category of systems $(\mathcal{M}_n)_{{n\geq1}}$ of {finitely presented} crystals over $X_n$ endowed with isomorphisms $(\mathcal{M}_{n+1})|_{X_n}\iso \mathcal{M}_n$. The category of \emph{isocrystals} over $\mathfrak{X}$ is the isogeny category $\mathrm{Isoc}(\mathfrak{X})\coloneqq {\mathrm{Crys}(\mathfrak{X})\otimes_{\Zp}\Qp}$.

\spa

The main objective of this project is to study the $I$-adic completion functor $$\mathrm{Isoc}(X)\to \mathrm{Isoc}(\mathfrak{X}).$$ 

\subsection{Problems}

\begin{prob}
	Suppose that $A$ admits a finite $p$-basis. Characterise $\mathrm{Crys}(\mathfrak{X})$ in terms of modules with {topologically quasi-nilpotent flat} connection over a flat lift $\tilde{A}$ over $\Zp$ of $A$ endowed with the $\tilde{I}$-adic topology, where $\tilde{I}$ is the preimage of $I$ in $\tilde{A}$. Construct crystals in $\mathrm{Crys}(\mathfrak{X})$ which do not come from crystals over $X$.
\end{prob}

\begin{prob}
	Suppose that $X$ is irreducible and Frobenius-smooth. Prove that $\mathrm{Isoc}(\mathfrak{X})$ is a Tannakian category over $\mathrm{Frac}(W(\kappa))$ where {$\kappa\coloneqq\bigcap_{n\geq0}A^{p^n}$ is the perfect field of constants of $A$}. This generalizes \cite[Cor. 3.3.3]{DrinfeldCrystals} to the formal setting.
\end{prob}

\begin{prob}
	For an isocrystal $\mathcal{M}\in \mathrm{Isoc}(X)$, study the relation between the monodromy group of $\mathcal{M}$ and the monodromy group of its image in $\mathrm{Isoc}(\mathfrak{X})$. If $\mathcal{M}$ {is the underlying isocrystal of a crystalline Dieudonné module}, show that the two groups coincide using \cite[Prop. 2.4.8]{deJ95}. Can one generalize this result to more general isocrystals? Are there isocrystals $\mathcal{M}\in \mathrm{Isoc}(\mathfrak{X})$ such that $F^*\mathcal{M}\simeq \mathcal{M}$, but $\mathcal{M}$ does not come from an isocrystal over $X$?
\end{prob}

\begin{prob}
	Develop a stacky approach to crystals and crystalline cohomology over $\mathfrak{X}$. In particular, extend \cite[Thm. 3.5.6]{DrinfeldCrystals} to $\mathfrak{X}$.
\end{prob}

\section{Monodromy groups beyond constant slopes}

\subsection{Setting}
Let $(A,\lambda)$ be a polarised abelian variety over $\ovfp$ of dimension $g$ and $(\mathcal{G},\lambda)$ its associated polarised $p$-divisible group. {Let $R$ be the complete local ring representing equicharacteristic deformations of $(\mathcal G,\lambda)$ and put $Z\coloneqq\Spec(R)$; its formal completion $\Spf(R)$ is the polarised deformation space and is non-canonically isomorphic to $\Spf(\ovfp[[x_1,\ldots,x_{g(g+1)/2}]])$. By the polarised Serre--Tate theorem, it also represents deformations of $(A,\lambda)$.} Denote its closed point by $z\in Z$. The aim of this project is to study the monodromy group of the universal isocrystal $\calM$ over $Z$. By \cite[Cor. 3.3.5]{OrdinaryHO} and \cite[Thm. II]{DvH22}, we know that the monodromy group of $\calM$ restricted to $\mathrm{Def}_\mathrm{sus}(\mathcal{G},\lambda)$\footnote{{This is the sustained deformation space in the sense of Chai--Oort: it is the largest locus on which the universal polarised $p$-divisible group is geometrically fibrewise constant.}} is the unipotent radical of the parabolic subgroup of $\mathrm{Sp}(\calM_z,\lambda)$ defined by the slope filtration. The goal of this project is to understand the monodromy group of $\calM$ over the entire scheme $Z$.

\subsection{Problems}
Let $\mathrm{Aut}(\mathcal{G},\lambda)$ be the abstract automorphism group of $(\mathcal{G},\lambda)$.

\begin{prob}
	Define an action of $\mathrm{Aut}(\mathcal{G},\lambda)$ on $G(\calM,z)$ using the universality of the deformation space.
\end{prob}

\begin{prob}
	Let $\eta$ be the generic point of $Z$ and let $\bar{\eta}$ be a point over $\eta$ associated to an algebraic closure of the field of fractions of $Z$. Prove that $G(\calM_{\eta_\mr{sep}},\bar{\eta})$ coincides with the unipotent radical of the parabolic subgroup of $\mathrm{Sp}(\calM_{\bar{\eta}},\lambda)$ defined by the (ordinary) slope filtration of $\calM_{\bar{\eta}}$. More generally, if $Z'$ is the closure of a central leaf in $Z$, compute $G(\calM_{\eta'_{\mr{sep}}},\bar{\eta}')$, where $\eta'$ is the generic point of $Z'$. This is an extension of \cite[Thm. II]{DvH22}.
\end{prob}

\begin{prob}[Extra]
	With the notation as in the previous problem, what can be said about the monodromy group $G(\calM_{\eta'},\bar{\eta}')$?
\end{prob}

\begin{prob}
	Suppose $(\mathcal{G},\lambda)$ supersingular. In this case, $(\mathcal{G},\lambda)$ has a large automorphism group. Use {Problem B.2.1 and Problem B.2.2} to prove that $G(\calM,z)$ is as large as possible, i.e. $$G(\calM,z)=\mathrm{Sp}(\calM_z,\lambda).$$ Start with the superspecial case if necessary.
\end{prob}

\begin{prob}
	Go beyond the supersingular case. Compute the monodromy groups when $(\mathcal{G},\lambda)$ is of height $4$ and slopes $0,1/2,1/2,1$ or when $(\mathcal{G},\lambda)$ is of height $6$. Try to formulate a general conjecture. In particular, is $G(\calM,z)$ connected?
\end{prob}

\end{appendix}

\end{document}